\documentclass[12pt]{amsart}
 \usepackage{amssymb,latexsym} 
\usepackage[all]{xy}
\newtheorem{thm}{Theorem}[section] 
\newtheorem*{thm*}{Theorem}
 \newtheorem{proposition}[thm]{Proposition} 
\newtheorem{lemma}[thm]{Lemma}
\newtheorem{remark}[thm]{Remark} 
\newtheorem{cor}[thm]{Corollary} 
\theoremstyle{definition}
 \newtheorem{defn}[thm]{Definition} 
\newtheorem{ex}[thm]{Example} 
\newtheorem{notation}[thm]{Notation} 

\newcommand{\by}[1]{\stackrel{#1}{\longrightarrow}}

\newcommand{\rank}{{\rm rank}\,}
  
\newcommand{\Hom}{{\rm Hom}\,}

\newcommand{\boxtensor}{{\Box\kern-9.03pt\raise1.42pt\hbox{$\times$}}} 
\newcommand{\Z}{{\mathbb Z}} 
 
\newcommand{\tensor}{\otimes}

 \newcommand{\sL}{{\mathcal L}} 
  
\newcommand{\sO}{{\mathcal O}}

 \newcommand{\mg}{{\mathfrak g}} 
  
\renewcommand{\tilde}{\widetilde}  
\numberwithin{equation}{section} \newcounter{elno} 
 
 \newcounter{example}[section] 
\def\theexample{\thesection.\arabic{example}}

\begin{document}

\title[Frobenius pull backs of vector bundles in higher dimensions]
{Frobenius pull backs of vector bundles in higher dimensions}
\author{V. Trivedxi}
\date{}
\address{School of Mathematics, Tata Institute of Fundamental Research,
Homi Bhabha Road, Mumbai-400005, India}
\thanks{}
\subjclass{13D40}

\maketitle
\section{Introduction} 

Let $X$ be a nonsingular 
projective variety defined over an algebraically closed field $k$ of an 
arbitrary characteristic, and let $H$ be a very ample line bundle on $X$. 
Let 
$E$ be a torsion free sheaf on $X$. Then the notion of $E$ being stable 
(resp. semistable) is well known and studied. In case $E$ is not 
semistable, then one has the notion of Harder-Narasimhan filtration of 
$E$. 
In this paper, we discuss the behaviour of
Harder-Narasimhan
filtrations of torsion free sheaves on $X$, under Frobenius pull-backs.

 First we give a bound, in Lemma~\ref{l1}, on  
the instability degree of $F^*E$ in terms of the 
instability degree of $E$ and well defined invariants of $X$ and $E$.
In particular it proves a sharper version of a conjecture of X.Sun
 (Remark~3.13 in [S]), 
for $p\geq \rank~E+\dim~X-2$, if $\dim~X \geq 2$, 
and for all $p>0$, if $\dim~X = 1$.

As a corollary,  we 
give a generalization of a result of 
Shepherd-Barron (Corollary~$2^p$ in [SB]), to an arbitrary  higher dimensional
nonsingular  variety,  
for  
characteristic $p \geq  \rank~E +\dim~X -2$: {\it the instability degree of
 a Frobenius pull-back of a semistable torsion free sheaf $E$ 
 is bounded by $(1/p)(\rank(E))$(the slope of 
the destabilising sheaf of $\Omega^1_X$)}.

 Here we use  a  result of
Ilangovan-Mehta-Parameswaran [IPM] about low height representations.

In Lemma~1.8 [T2], the  author  has proved that, for 
a vector bundle $E$ over 
a curve $X$, the HN 
filtration of 
$F^{*}(E)$ is  a refinement of $F^{*}[\mbox{HN filtration of} E]$, 
under some assumptions on the characteristic of the base field.
 Here we extend this   
result, in Theorem~\ref{t1}, to  a variety of arbitrary dimension
 (again with some restrictions: 
on a
lower bound on $p$ in terms of rank $E$ and the slope of 
the destabilising sheaf of the 
cotangent bundle of  $X$).
Moreover, under the same hypothesis on $p$, we observe that  
(Corollary~\ref{cc1}) if the ranks of first and last proper subsheaves of
 the HN filtration of $E$ are equal to $1$ and $\rank~E-1$, respectively,
 then the instability degree does not change under
 any iterated  Frobenius pull back.

 We recall some examples \ref{ex} and \ref{exx} of Raynaud~[R] and Monsky~[M],
respectively,  
to show that  some lower bound  on the characteristic $p$ 
(in terms of both $\rank~E$ and $\deg~X$) is necessary, for Theorem~\ref{t1}.

One also observes (proved for curves in [T2]) that each normalised 
HN slope of a HN sheaf of $F^{s*}E$ is bounded in terms of the  slope of HN
 sheaf (of $E$), to 
which it `descends', and explicit invariants of $X$ and rank of $E$. 
Various other results proved for curves in [T2] are valid for higher dimnesion 
with some modification in the bounds. Since proofs are very
 similar to the case of curves, we have stated them without proofs. 

We would like to thank V.B. Mehta for useful discussions.

 \section{vector bundles} 
Let $X$ be a smooth projective variety of dimension $n$ over
 an algebraically closed field of 
characteristic $p > 0$. Let $E$ be  a torsion free sheaf of rank $r$ 
on $X$. We also fix a polarization $H$ of $X$. 

\begin{defn}\label{d1} A torsion-free sheaf $E$ 
 is $\mu$-semistable (with respect to the  polarization $H$), if
for all subsheaves $F \subset E$, one has
$$\mu(F) \leq \mu(E),~~\mbox{where}~~\mu(E) = (c_1(E)\cdot
H^{n-1})/ \rank(E).$$
\end{defn}

\begin{defn}\label{d2}For a torsion free sheaf $E$
on $X$, consider the Harder-Narasimhan filtration given by 
\begin{equation}\label{e11} 0 = E_0 \subset E_1\subset \cdots \subset E_l
\subset E_{l+1} = E.    
\end{equation}
Then, for $i\geq 1$, 
$$\mu_i(E) = \mu(E_i/E_{i-1})~~\mbox{and}~~ \mu_{max}(E) = 
\mu(E_1)~~~\mbox{and}~~~ \mu_{min}(E) = 
\mu(E/E_l).$$  
The {\it instability degree} $I(E)$  of $E$ is defined as 
$I(E) = \mu_{max}(E)-\mu_{min}(E)$.
\end{defn}

\begin{defn}\label{d6}If $X$ is a projective variety defined over an   
algebraically closed field of characteristic $p>0$, then the absolute
Frobenius
morphism $F:X\to X$ is a morphism of schemes which is identity on the
underlying set of $X$ and on the underlying sheaf of
rings  $F^{{\#}}:\sO_X\to \sO_X$ is the $p^{th}$
power map. \end{defn}

We recall the following well known  
\begin{lemma}\label{l3} If  
 $E_1$ and $E_2$ are two torsion free sheaves on $X$ then 
$$\mu_{min}(E_1) > \mu_{max}(E_2) \implies \Hom_{\sO_X}(E_1, E_2) = 0.$$ 
\end{lemma}

We recall the following result (see [SB], Proposition~1).

\begin{lemma}\label{l0}Let $E$
 be a semistable torsion free sheaf on $X$ such 
that 
$F^*E$ is not semistable. Let 
$$0 = F_0 \subset F_1\subset \cdots \subset F_l
\subset F_{l+1} = F^*E $$
be the Harder-Narasimhan filtration. Then there exists a canonical 
connection $\triangledown_{can}:F^*E\rightarrow F^*E\tensor \Omega^1_X$  
such that, for every $1\leq i \leq l$, the $\sO_X$-homomorphisms
$F_i\rightarrow (F^*E/F_i)\tensor \Omega^1_X$ induced by $\triangledown_{can}$ 
are 
nontrivial.\end{lemma}

\begin{notation}\label{n1} Let $l(E)$ denote the number of nontrivial subsheaves in
the HN filtration of $E$, {\it e.g.}, if $E$ has the HN filtration 
as in  Definition~\ref{d2} above, then  
$l(E) = l$.

Moreover $s(X,E)$ denotes the number of $\triangledown_{can}$-invariant subsheaves of
$F^*E$, which occur properly in the HN filtration of $F^*E$. We recall that 
$\triangledown_{can}$-invariant subsheaves 
 $F^*E$ are precisely those which descend to a subsheaf of $E$. Clearly 
$0\leq s(X,E)\leq l(F^*E)$.
\end{notation}

Now we recall the following conjecture of X. Sun (Remark~(3.13) of [S]):
$$I(F^*E) \leq (l-s)\mu_{max}(\Omega^1_X) + psI(E), $$
where $l$ and $s$ are as in Lemma~\ref{l1} below. The following Lemma is a
 modified version of this conjecture. 
In particular, it implies
 the conjecture, if $p\geq r+n-(s+2)$.
 The proof is by refining some arguments  given in the
 proof of Theoren~3.4 in [S]. 

\begin{lemma}\label{l1}Let $\dim X = n$ and $\rank E 
= r$ and let $\Omega^1_X$ denote the cotangent bundle of $X$. 
Let $l = l(F^*E)$ and let $s = s(X,E)$. 
Suppose $p 
\geq r+n-(s+2)$. Then 
$$I(F^*E) \leq (l-s)~\mu_{max}(\Omega^1_X) + \epsilon\cdot pI(E),$$
where $\epsilon = \mbox{min}\{1,s\}$.
\end{lemma}

\begin{proof} If $F^*E$ is semistable then it is obvious.
Suppose $F^*E$ is not semistable, then we have the Harder-Narasimhan  
filtration of $F^*E$,
$$0= F_0\subset F_1 \subset \cdots \subset F_l\subset F_{l+1} = F^*E,$$
in particular $l\geq 1$.
Let $ 0 \subset E_1 \subset \cdots \subset E_t \subset E$
be the HN filtration of $E$. 
Let 
$$ S = \{i\mid 1\leq i \leq l, \mbox{where}~~F_i 
~\mbox{descends to some}~~ E_{j_i}\}.$$

Now 
$$I(F^*E) = \sum_{i=1}^{l}\mu(F_i/F_{i-1})-\mu(F_{i+1}/F_i)$$
$$ = \sum_{i\notin S}\mu(F_i/F_{i-1})-\mu(F_{i+1}/F_i) +
 \sum_{i\in S}\mu(F_i/F_{i-1})-\mu(F_{i+1}/F_i)$$

\noindent{Case~(1)}. Suppose $i\notin S$. Then 
 we have a nonzero 
$\sO_X$-linear map
$$\sigma_i: F_i \rightarrow (F^*E/F_i)\tensor\Omega^1_X.$$
Let $$0=M_0\subset M_1\subset \cdots \subset M_{m+1} = \Omega^1_X$$
be the HN filtration of the cotangent bundle $\Omega^1_X$ of $X$.
Note that $F_i/F_{i-1}$, $F^*E/F_i$
and $M_i/M_{i-1}$ are locally free sheaves of $\sO_U$-modules, where $U$ 
is  an open subscheme such that $X\setminus U$ is of codimension 
$\geq 2$ in $X$.

Let $j$ be the minimum integer such that $\sigma_i(F_i) \subseteq \frac{F^*E}{F_i}\tensor M_j$.
This induces a nonzero $\sO_X$-linear map 
$$F_i \rightarrow (F^*E/ F_i)\tensor M_j/M_{j-1}.$$
Hence, by Lemma~\ref{l3},
\begin{equation}\label{e1}\mu_{min}(F_i)\leq \mu_{max}((F^*E/F_i)\tensor 
(M_j/M_{j-1})).\end{equation}
 We note that $\rank(F_{i+1}/F_{i}) \leq r-(s+1)$. Therefore, by [IMP], 
$\frac{F_{i+1}}{F_i}\tensor\frac{M_j}{M_{j-1}}$ is semistable, as 
$$p+1 \geq  r-(s+1) + n \geq \dim\left(\frac{F_{i+1}}{F_i}\right) + 
\dim\left(\frac{M_j}{M_{j-1}}\right).$$
 Hence   

$$\left(\frac{F_{i+1}}{F_i}\tensor\frac{M_j}{M_{j-1}}\right)\subset \cdots 
\subset
\left(\frac{F_{t}}{F_i}\tensor\frac{M_j}{M_{j-1}}\right)\subset 
\left(\frac{F_{t+1}}{F_i}\tensor\frac{M_j}{M_{j-1}}\right) = 
\left(\frac{F^*E}{F_i}\tensor\frac{M_j}{M_{j-1}}\right)$$

is the HN filtration of $\frac{F^*E}{F_i}\tensor\frac{M_j}{M_{j-1}}$.
Therefore 
$$\begin{array}{lcl}
\mu_{max}((F^*E/F_i)\tensor (M_j/M_{j-1})) & = & 
\mu(\frac{F_{i+1}}{F_i}\tensor\frac{M_j}{M_{j-1}})\\
& = &  \mu(\frac{F_{i+1}}{F_i}) + \mu(\frac{M_j}{M_{j-1}})\\
&  \leq & \mu(\frac{F_{i+1}}{F_i}) + \mu_{max}(\Omega^1_X)
\end{array}$$
Hence, equation~(\ref{e1}) implies that 
$$\mu(F_i/F_{i-1})-\mu(F_{i+1}/F_{i})\leq \mu_{max}(\Omega^1_X).$$
Therefore, for $i \notin S$,  
$$\sum_{i\notin S}\mu(F_i/F_{i-1})-\mu(F_{i+1}/F_i) \leq  
(l-s)~\mu_{max}(\Omega^1_X).$$

\noindent{Case~(2)}.\quad Let $F_{i_1},  \ldots, F_{i_s}$ be 
the subsheaves of the HN filtration of $F^*E$ which  
descend to subsheaves $E_{j_1}, \ldots, E_{j_s}$ of $E$,  where 
$E_{j_1}\subset \cdots \subset E_{j_s}$.

\noindent{Claim}.\quad Let $E_{j_0} = (0)$ and $E_{j_{s+1}} = E$, then we have
$$\mu\left(\displaystyle{\frac{F^*E_{j_k}}{F_{(i_k)-1}}}\right)
-\mu\left(\displaystyle{\frac{F_{(i_k)+1}}{F^*E_{j_k}}}\right)
\leq p\mu\left(\displaystyle{\frac{E_{j_k}}{F_{j_{(k-1)}}}}\right)
-p\mu\left(\displaystyle{\frac{E_{j_{(k+1)}}}{E_{(j_k)}}}\right)
$$
\noindent{Proof of the claim}:\quad 
We note that  $F^*E_{j_k}/F^*E_{j_{(k-1)}}$ has the following HN filtration 
$$ \frac{F^*E_{j_{(k-1)}}}{F^*E_{j_{(k-1)}}} \subset \cdots \subset 
\frac{F_{(i_k)-1}}{F^*E_{j_{(k-1)}}} \subset
 \frac{F^*E_{j_k}}{F^*E_{j_{(k-1)}}}.$$
Therefore 
$$\mu_{min}\left(\frac{F^*E_{j_k}}{F^*E_{j_{(k-1)}}}\right) = 
\mu\left(\displaystyle{\frac{F^*E_{j_k}}{F_{(i_k)-1}}}\right)
\leq 
\mu\left(\displaystyle{\frac{F^*E_{j_k}}{F^*E_{j_{(k-1)}}}}\right)
= p\mu\left(\displaystyle{\frac{E_{j_k}}{E_{j_{(k-1)}}}}\right),$$

where the second inequality follows because 
$\frac{F^*E_{j_k}}{F^*E_{j_{(k-1)}}}$ is a quotient of 
$\frac{F^*E_{j_k}}{F_{(i_k)-1}}$ as $E_{j_{(k-1)}} \subseteq F_{(i_k)-1}$. 
On the other hand 
$$\displaystyle{\frac{F_{(i_k)+1}}{F^*E_{j_k}}}
\subset \cdots \subset 
\displaystyle{\frac{F^*E_{j_{(k+1)}}}{F^*E_{j_k}}}
$$ is the  HN filtration of 
$\displaystyle{\frac{F^*E_{j_{(k+1)}}}{F^*E_{j_k}}}$. Therefore 
$$\mu_{max}\left(\displaystyle{\frac{F^*E_{j_{(k+1)}}}{F^*E_{j_k}}}\right)
= \mu\left(\displaystyle{\frac{F_{(i_k)+1}}{F^*E_{j_k}}}\right)
\geq \mu\left(\displaystyle{\frac{F^*E_{j_{(k+1)}}}{F^*E_{j_k}}}\right)
 = p\mu\left(\displaystyle{\frac{E_{j_{(k+1)}}}{E_{j_k}}}\right).$$
This proves the claim.

Now 
$$\displaystyle{\sum_{i\in S = \{i_1, i_2, \ldots, i_s\}}}
\mu\left(\displaystyle{\frac{F_i}{F_{i-1}}}\right)-
\mu\left(\displaystyle{\frac{F_{i+1}}{F_i}}\right)$$
$$\begin{array}{cl}
 \leq & 
p\left[\mu(E_{j_1})-\mu\left(\displaystyle{\frac{E_{j_2}}{E_{j_1}}}\right)
+\mu\left(\displaystyle{\frac{E_{j_2}}{E_{j_1}}}\right)
-\mu\left(\displaystyle{\frac{E_{j_3}}{E_{j_2}}}\right)+\cdots +
\mu\left(\displaystyle{\frac{E_{j_s}}{E_{j_{s-1}}}}\right)
-\mu\left(\displaystyle{\frac{E}{E_{j_s}}}\right)
 \right]\\
 = & p\left[\mu(E_{j_1})-\mu\left(\displaystyle{\frac{E}{E_{j_s}}}\right)\right]\\
 \leq & p\left[\mu_{max}(E)-\mu_{min}(E)\right] = p~I(E),\end{array}$$
for a proof of  the second last inequality one can 
see Proposition~3.3~(3) in [S]. 
Now the lemma follows at once by Case~(1) and Case~(2). \end{proof}

Now,  modifying  the proof of Theorem~3.12 of [S] with similar 
arguments, we get the 
following (here $l$ and $s$ are as in Theorem~3.12 of [S])
\begin{thm}\label{t5} Let $\dim~X = n$ and $\rank~E = r$ and let $\Omega_X^1$ 
denote the cotangent bundle of $X$. Then, for any $p>0$, we have 
$$L_{max}(E) -L_{min}(E) \leq \frac{l-s}{p}\cdot L_{max}(\Omega_X^1)+I(E),$$
where 
$$L_{max}(E) :=\mbox{lim}_{k\to\infty}\frac{\mu_{max}(F^{k*}E)}{p^k},~~~
L_{min}(E) :=\mbox{lim}_{k\to\infty}\frac{\mu_{min}(F^{k*}E)}{p^k}.$$
In particular, 
$$I(F^*E) \leq (r-1)L_{max}(\Omega_X^1) + I(E),~~~\mbox{if}~~ 
\mu_{max}(\Omega_X^1) > 0$$
otherwise 
$I(F^*E) = I(E)$.
\end{thm}

The following corollary is a generalization of a Corollary~$2^p$ of 
[SB] (there it is proved for $\dim~X =1$ and for every prime $p$).

\vspace{5pt}

\begin{cor}\label{c1}With the notation as in Lemma~\ref{l1}, if in addition
$E$ is semistable then 
$$I(F^*E) \leq (r-1)\mu_{max}(\Omega^1_X).$$\end{cor}

\noindent{\bf Remark}.\quad In the case $\dim~X = 1$, the bundle
 $M_j/M_{j-1} = \Omega^1_X$ is  a line bundle and therefore 
$\frac{F_{i+1}}{F_i}\tensor\frac{M_j}{M_{j-1}}$ is semistable, for every
 prime $p$.
 In particular, we have the following 

\vspace{5pt}

 \begin{proposition}\label{p5}\quad{\it Let $X$ be a smooth 
projective curve of 
genus $g\geq 1$ and $E$ be a torsion free sheaf on $X$. 
Then, for any prime $p$, we have  
$$I(F^*E) \leq (l-s)(2g-2) + \epsilon\cdot pI(E),$$}
 where $\epsilon = \mbox{min}\{1,s\}$.
\end{proposition}
\vspace{5pt}

This  gives a sharper bound than given in Theorem~3.4 of [S].   
In fact the following examples show that perhaps this is the 
 optimal bound on the instability degree of $F^*E$: First we recall
 some results from [M] and [T1].
 Let $X$ be a nonsingular plane curve of degree $d$. 
Therefore $X= \mbox{Proj}~R$, where 
$R= k[x,y,z]/(h)$, with $h$  a homogeneous polynomial of degree $d$ 
and $k$ is an algebraically closed field of characteristic $p$.
Consider the canonical map 
$$0\longrightarrow V\longrightarrow H^0(X, \sL)\tensor\sO_X \longrightarrow
 \sL\longrightarrow 0,$$
where $\sL$ is the very ample bundle induced by $X\hookrightarrow {\bf P}^2$.
Then, by Corollary~5.4 of [T1], the Hilbert-Kunz multiplicity of $(X, \sL)$ 
(which is same as the Hilbert-Kunz multiplicity of $R$ with 
respect of the ideal $(x,y, z)$) is given as  
$$e_{HK}(X,\sL) = \frac{3d}{4} + \frac{l_1^2}{4dp^{2s_1}},$$ 
where 
$s_1\geq 1$ such that 
$F^{(s_1-1)*}V$ is semistable and $F^{s_1*}V$ is not semistable. and $l_1$ is
 an integer such that $l_1\equiv pd~(mod~2)$ and $0\leq l_1 \leq d(d-3)$.
Moreover, by the proof of Theorem~5.3 of [T1], we have an interpretation of
  $l_1$ as
 follows: Let $\sL_1 \subset F^{s_1*}V$ be the HN filtration of $F^{s_1*}V$
 then 
$\deg~\sL_1 = -(d/2)p^{s_1}+(l_1/2)$. In particular  
$I(F^{s_1*}V) = l_1$.

On the other hand, Monsky in [M], using a theorem of C.Han calculated Hilbert-Kunz 
multiplicity of various  irreducible trinomial
 plane curves. Here we recall two of those examples.
Let $R = k[x,y,z]/(h)$, where 
\begin{enumerate}
\item $h = x^{d-1}y+y^{d-1}z+z^{d-1}x$
 and $d\geq 4$  is an even integer and $p$ is a prime number 
such that $p\equiv\pm(d-1)~(mod~2(d^2-3d+3))$,
\item $h = x^d+y^d+z^d$ and $d$ is an even integer and $p$ is a prime number 
such that $p\equiv d\pm~1(mod~2d)$
\end{enumerate}
(note that, for any given $d$, there are infinitely many primes satisfying 
conditions in (1) and (2))
then $X=\mbox{Proj}~R$ is a nonsingular projective plane curve with  
$$e_{HK}(X, \sL) = \displaystyle{\frac{3d}{4} + \frac{(d(d-3))^2}{4dp^2}}.$$ 
Now,  Corollary~5.4 of [T1], stated above, implies that
  in these two examples  $s_1 = 1$ and $l_1 = d(d-3)$.
  
Therefore, by Theorem~5.3 of [T1], we have
$I(F^*V) = d(d-3) = 2g-2$ and  $l= l(X,F^*V)=1$ and 
$s= s(X, E) = 0$. In particular 
$$I(F^*V) = (l-s)(2g-2)+ pI(V),$$ for 
infinitely many primes. On the other hand, in example~($2$) above 
(see [HM]), for $d = 4$,
  $I(F^*E) = 0$ for infinitely many primes.

\vspace{5pt}

\noindent{\bf Remark}.\quad  Let $X$ be a nonsingular projective variety of 
dimension $n$, 
over a field $k$.  Let  
$i:X\hookrightarrow {\bf P}^{n_0}$ be an closed embedding (we can always take 
$n_0 = 2\dim~X+1$). 
This gives a surjective 
map of sheaves of $\sO_X$-modules 
$$\Omega^1_{{{\bf P}^{n_0}_k}/k}\tensor\sO_X \longrightarrow 
\Omega^1_{X/k}.$$
Let $M_1$ be a subsheaf of $\Omega^1_X$  such that $\mu_{max}(\Omega^1_X) = 
\mu(M_1)$.
Then the following composite map of sheaves of $\sO_X$-modules
$$ \Omega^1_{{{\bf P}^{n_0}_k}/k}\tensor\sO_X(2) \longrightarrow 
\Omega^1_{X/k}(2) \longrightarrow (\Omega^1_{X/k}/M_1)(2) $$
is surjective. In particular, the sheaf  $(\Omega^1_{X/k}/M_1)(2)$ is generated
 by global sections as $\Omega^1_{{{\bf P}^{n_0}_k}/k}(2)$ is so. 
Therefore $\deg(\Omega^1_{X/k}/M_1)(2)\geq 0$. In particular 
$\deg(M_1)(2) \leq \deg(\Omega^1_X)(2)$, which gives the inequality  
$$\mu(M_1) \leq \frac{n}{\rank(M_1)}\mu(\Omega^1_X)+
\deg(\sO_X(2))\frac{n-\rank(M_1)}{\rank(M_1)},$$
Note that $1\leq \rank(M_1) \leq n$, and the other invariants on the right hand 
side of the inequality  depend on $\deg~X$ 
(with respect to the embedding $i$) and invariants of $X$, namely 
$\mu(\Omega^1_X)$ and $\dim~X$.

\vspace{5pt}

\noindent{\bf Remark}.\quad 
We recall Lemma~1.5 of [T2] with a small modification.

\begin{lemma}\label{l2} Let $E$ be a  torsion free coherent sheaf on 
$X$, where $\dim~X \geq 1$. Let $r = \rank~E$.
Suppose $E$ is not semistable. Then, for the  HN filtration 
(\ref{e11}) of $E$, we have 
$$\mu_i(E)-\mu_{i+1}(E) \geq \frac{4}{r^2}, ~~\mbox{for every}~~1\leq i 
\leq l.$$
 \end{lemma} 
\begin{proof}Let 
$$0 = E_0 \subset E_1\subset \cdots \subset E_l
 \subset E_{l+1} = E$$
be the HN filtration of $E$. Let $d_i = \deg(E_i/E_{i-1})$ 
and $r_i = \rank(E_i/E_{i-1})$. Then 
$$\mu_i(E)-\mu_{i+1}(E) = \frac{d_i}{r_i}-\frac{d_{i+1}}{r_{i+1}}
\geq \frac{1}{r_ir_{i+1}} \geq \frac{4}{(r_i+r_{i+1})^2}\geq \frac{4}{r^2},$$
since $r_i+r_{i+1}\leq r$.
This proves the lemma.
\end{proof}

The 
following theorem is a generalization of  Lemma~1.8 
in [T2], where it is proved for $\dim~X = 1$.

\begin{thm}\label{t1} Let $p \geq \max\{(r+n-2),~~\mu_{max}(\Omega^1_X) 
(r^3/4)\}$, where 
$E$ is a torsion free coherent sheaf over $X$.  
Let 
$$0 \subset E_1\subset E_2 \subset \cdots  \subset E_l \subset E$$
 be the 
HN 
filtration of $E$.
Then 
$$F^*E_1\subset F^*E_2 \subset \cdots  \subset F^*E_l \subset F^*E $$
is a subfiltration of the HN filtration  of $F^*E$, {\it i.e.}, if 
$$0\subset {\tilde E_1}\subset {\tilde E_2} \subset \cdots  \subset 
{\tilde E_{l_1}} \subset {\tilde E_{l_1+1}} = F^*E $$
is the HN filtration of $F^*E$ then, for every $1\leq i\leq l$ there 
exists $1\leq j_1\leq l_1$ such that $F^*E_i = {\tilde E_{j_i}}$. 
\end{thm}
\begin{proof}Due to Corollary~\ref{c1}, the arguments, as given in the 
proof of Lemma~1.8 in 
[T2], can be adapted directly to the higher dimensional variety $X$. Hence 
the theorem follows.\end{proof}

However Theorem~\ref{t1}, above, cannot be generalised to arbitrarily small 
prime characteristics $p>0$. For 
this  one constructs the  following counterexamples from
  examples due to Raynaud~[R] and Monsky~[M].

\begin{ex}\label{ex}Let $X$ be a nonsingular projective curve of genus $ g 
= pk+1$ 
defined over an algebraically closed  field of char~$p > 2$, where $k$ is 
any 
positive integer. 
Consider the canonical map of locally free sheaves of $\sO_X$-modules
$$0\rightarrow \sO_X\rightarrow F_*\sO_X \rightarrow B\rightarrow 0,$$
then, by Theorem~4.1.1 of [R], the vector bundle $B$ is semistable of 
rank~$p-1$ and $\mu(B) = g-1$.

Moreover, by Remark~4.1.2 of [R], the Frobenius pull back $F^*B$ has the 
HN filtration given as follows: 
\begin{equation}\label{er}
 0=B_p \subset B_{p-1} \subset B_{p-2} \subset 
\cdots \subset B_2 
\subset  
B_1 = F^*B,\end{equation}
where $B_i/B_{i+1} \simeq {\Omega_X^{\tensor{i}}}$, for all $1\leq i \leq 
p-1$.
Now we take a line bundle  $L$ on $X$ of 
degree $ d = 2k(p-1)$. 
Let $V = L \oplus B $ then  
$ 0\subset L \subset 
V$ is the HN filtration
of $V$,  as 
$$\mu(L) = 2k(p-1) > pk = g-1 = \mu(B) = \mu(V/L)$$ 
and $L$ and $V/L \simeq B $ are semistable 
vector bundles. On the other hand, one can check that the filtration  
$$ 0\subset V_{p-1} = F^*L 
\oplus B_{p-1}\subset V_{p-2} = F^*L\oplus B_{p-2} \subset \cdots \subset
V_1 = F^*L\oplus B_1 = F^*V $$ is the 
HN filtration  of $F^*V$. 
\end{ex}

\begin{ex}\label{exx}Now we come back to Monsky's example of trinomial 
curves $h=x^d+y^d+z^d$ (see the discussion following
 Proposition~\ref{p5} above),
 with conditions on $d$ and $p$ as before.
Here, for the syzygy bundle $V$, the HN filtration of $F^*V$ is given by 
$0\subset\sL_1\subset F^*V$, where 
$$\deg~\sL_1 = -\frac{dp}{2}+\frac{d(d-3)}{2}~~~\mbox{and}~~
\deg~F^*V = -\frac{dp}{2}.$$
We choose a line bundle $\sL_0$ such that $\deg~\sL_0 = \frac{d}{2}+\delta $,
 where 
$\delta$ is an integer such that $1\leq \delta\cdot p \leq (d(d-3))/2$,
{\it e.g.}, if $p\geq 7$ a prime number then $d=p+1$ and $\delta = 1$ 
will satisfy all these conditions.
Let $W = \sL_0\oplus V$, then it is easy to check that 
 $0\subset \sL_0\subset W$ is the HN filtration of the vector bundle $W$ and 
$$0\subset \sL_1\subset \sL_1\oplus F^*\sL_0 \subset F^*V\oplus F^*\sL_0 
= F^*W$$ 
is the HN filtration of $F^*W$.
Therefore  HN filtration of $F^*W$ is not a refinement of the pull 
back of HN filtration of $W$. 
 
In particular the statement of Theorem~\ref{t1} is not true in general 
even for curves, for smaller
(compared to the genus of the curve or rank of the vector bundle)
characteristics. 
\end{ex}

\begin{cor}\label{cc1} Let $p \geq \max\{(r+n-2),~~\mu_{max}(\Omega^1_X) 
\frac{r^3}{4}\}$, where 
$E$ is a torsion free sheaf, of rank $r$, over $X$. Suppose
 $E$ is not semistable. 
Let 
$$0 \subset E_1\subset E_2 \subset \cdots  \subset E_l \subset E$$
 be the 
HN 
filtration of $E$.
Then $$I(F^*E) \leq \mu_{max}(\Omega^1_X)(
\rank\left(\frac{E}{E_l}\right)+\rank(E_1)-2) + pI(E).$$

In particular, if $E$ has HN filtration such that $\rank(E_1) = \rank(E/E_l)
= 1$ then, for every $s\geq 1$, we have 
$$I(F^{s*}E) = p^sI(E).$$  
\end{cor}
\begin{proof}By Theorem~\ref{t1}, the HN filtration of $F^*E$ is of the form 
$$ 0\subset E_{01}\subset \cdots \subset E_{0t_0} \subset 
\cdots \subset F^*E_i \subset E_{i1}\subset \cdots \subset E_{it_i}
\subset F^*E_{i+1}\subset \cdots \subset F^*E.$$
Therefore, for every $0\leq i\leq l$, the sheaf $E_{i+1}/E_{i}$ is semistable
and 
$$ 0\subset \frac{E_{i,1}}{F^*E_i}\subset \frac{E_{i,2}}{F^*E_i}\subset 
\cdots \subset \frac{E_{i,t_i}}{F^*E_i}\subset \frac{F^*E_{i+1}}{F^*E_i} $$
is the HN filtration of $F^*E_{i+1}/F^*E_i$.
Therefore, by Lemma~\ref{l1}, 
\begin{equation}\label{ee1}
\mu_{max}(F^*E_1)-\mu_{min}(F^*E_1) \leq \mu_{max}(\Omega^1_X)(\rank(E_1)-1)
 \end{equation}
Similarly 
\begin{equation}\label{ee2}
\mu_{max}\left(\frac{F^*E}{F^*E_l}\right)-
\mu_{min}\left(\frac{F^*E}{F^*E_l}\right) \leq 
\mu_{max}(\Omega^1_X)(r-\rank(E_l)-1)
 \end{equation}
But, by construction, it follows that  
$$\mu_{max}(F^*E_1) = \mu_{max}(F^*E)~~\mbox{and}~~
 \mu_{min}\left(\frac{F^*E}{E_l}\right) = 
\mu_{min}(F^*E).$$
 Therefore, by equations~(\ref{ee1}) and (\ref{ee2}), we get 
\begin{equation}\label{ee3}
\mu_{max}(F^*E)-\mu_{min}(F^*E)\leq 
\mu_{max}(\Omega^1_X)(r-\rank\left(\frac{E_l}{E_1}\right)-2)+\mu_{min}(F^*E_1)
-\mu\left(\frac{E_{l,1}}{F^*E_l}\right).\end{equation}
But 
$$\mu_{min}(F^*E_1)\leq \mu(F^*E_1) = p\mu(E_1) = p\mu_{max}(E)$$
and 
$$\mu\left(\frac{E_{l,1}}{F^*E_l}\right) = \mu_{max}\left(\frac{F^*E}{F^*E_l}\right)
 \geq \mu\left(\frac{F^*E}{F^*E_l}\right) = p\mu\left(\frac{E}{E_l}\right) = 
p\mu_{min}(E).$$
Therefore the 
right side of  Equation~(\ref{ee3})
$$\leq 
 \mu_{max}(\Omega^1_X)(r-\rank\left(\frac{E_l}{E_1}\right)-2)
+ p(\mu_{max}(E)-\mu_{min}(E)).$$

Now if $\rank(E_1) =\rank(E/E_l) = 1$ then 
$I(F^{*}E) = pI(E)$. Since the HN filtration of $F^{s*}E$ is a refinement of 
the Frobenius pull back of the HN filtration of $F^{(s-1)*}E$, 
the first  subsheaf and the last quotient sheaf in the HN 
filtration of $F^{(s-1)*}E$ are of rank $ = 1$, for every $s\geq 1$. Hence 
$I(F^{s*}E) = p^sI(E)$.
This proves the corollary.\end{proof}

\noindent{\bf Remark}\quad
In fact, Lemma~\ref{l5} below implies that, for any $s\geq 1$,
the normalised HN slopes 
of $F^{s*}(E)$ can be  estimated in terms of the  HN slopes of $E$ and a 
bounded 
constant. The 
proofs of Lemma~\ref{l5}, Proposition~\ref{pp8} and Proposition~\ref{pp1}
are along the same
 lines as in Lemma~1.14, Proposition~1.16 and Proposition~2.2 of [T2], 
respectively; we omit the details.

\begin{defn}\label{d4} Let $E$ be a torsion free sheaf  on $X$. A
subsheaf  $F_j\neq 0$ occuring in the HN filtration of $F^{s*}E$ is
said
to {\it almost descend} to a sheaf $E_{i}$ occuring in the HN
filtration of
$E$ if $F_j\subseteq F^{s*}E_{i}$ and  $E_{i}$ is the
smallest subsheaf in the
HN filtration of $E$, with this property.
\end{defn}

\begin{remark}\label{rr1}Henceforth we assume that the characteristic
$p$ satisfies
 $$p\geq \mbox{max}\{r+n-2, \mu_{max}(\Omega^1_X)(r^3/4)\}.$$
\end{remark}

\begin{lemma}\label{l5}Let $E$ be a torsion free sheaf  on $X$ of rank $r$.
 Let 
$F_j\neq 0$ be a subsheaf in the HN filtration of $F^{s*}E$, which
almost descends to a sheaf $E_i$ occuring in the HN filtration
of 
$E$. Then 
$$ \frac{{\mu}_j(F^{s*}E)}{p^s} = {\mu}_i(E) +\frac{C}{p},$$
where $|C|\leq 2|\mu_{max}(\Omega^1_X)|(r-1)$.\end{lemma}

\begin{notation}\label{n3} We  fix  a torsion free sheaf  $V$ on $X$ of
 rank $r$ with the HN filtration
$$0 ={E_0} \subset { E_1}\subset {E_2} \subset \cdots
\subset {
E_l}\subset { E_{l+1}} = V.$$
Let
\begin{equation}\label{e7}
0\subset F_1\subset F_2\subset \cdots \subset F_t
\subset F_{t+1} =
F^{k*}V\end{equation}
be the HN filtration of $F^{k*}V$, and let
$${r_i}(F^{k*}V) = \rank\left(\frac{F_i}{F_{i-1}}\right)
 ~~~\mbox{and}~~a_i(F^{k*}V) =
\frac{\mu_i(F^{k*}V)}{p^k}.$$

Moreover, we
 choose  an integer $s\geq 0$  such that
$F^{s*}(V)$ has a strongly semistable HN filtration and we denote
$${\tilde a_i}(V) = a_i(F^{s*}(V))~~\mbox{and}~~ {\tilde r_i}(V) =
r_i(F^{s*}(V))$$
(note that, by Theorem~\ref{t1},
these numbers are
independent of the choice of such an $s$).

\end{notation}

\begin{proposition}\label{pp8}With the notation as above and the 
hypothesis on 
$p$ as in Remark~\ref{rr1},
if a subsheaf  $F_j$ of the HN filtration of $F^{k*}V$
almost descends to a
subsheaf ${E_i}$ of the HN filtration of $V$ then,
for any $m\geq 1$,
$$a_j(F^{k*}V)^m = \mu_i(V)^m + \frac{C}{p},$$
where $|C|\leq 4|\mu_{max}(\Omega^1_X)|(r-1)({\mbox max}\{2|\mu_1(V)|, \ldots, 2|\mu_{l+1}(V)|,
2\}^{m-1})$.\end{proposition}

\begin{proposition}\label{pp1} Let $f: X_A\by{} {\rm Spec}~A$ be
a projective morphism of Noetherian schemes, smooth of relative
dimension~n, where $A$ is a finitely
generated $\Z$-algebra and is an integral domain. Let $\sO_{X_A}(1)$ be an
$f$-very ample invertible sheaf on $X_A$. Let $V_A$ be a torsion free
sheaf on $X_A$. For  $s\in {\rm Spec}~A$, let $V_s= V_A\tensor_A
{\overline {k(s)}}$ be the induced torsion free sheaf on the smooth projective
variety $X_s = X_A\tensor_A{\overline{k(s)}}$.
Let $s_0 = {\rm Spec}~Q(A)$ be the generic point of  ${\rm
Spec}~A$.
Then,
\begin{enumerate}
\item for any $k\geq 0$ and $m\geq 0$, we have
$$\lim_{s\to s_0}\sum_j{r_j}(F^{k*}V_s) a_j(F^{k*}V_s)^m =
\sum_i{r_i}(V_{s_0})\mu_i(V_{s_0})^m.$$
\item Similarly
$$\lim_{s\to s_0}\sum_j{\tilde r_j}(V_s) {\tilde a_j}(V_s)^m =
 \sum_i{r_i}(V_{s_0})\mu_i(V_{s_0})^m,$$
\end{enumerate}
where in both the limits, $s$ runs over closed points of ${\rm{Spec}}~A$.
\end{proposition}

\section{Some more generalities}

Let $X$ be a smooth projective variety over an algebraically closed field
$k$. Let $H$ be an ample line bundle on $X$.  

 Analogous to Theorem~\ref{t1}, which is given for vector bundles,  we
prove the following result for
principal $G$-bundles, 
in the light of Proposition~3.4 of [MS]
(we follow the same notation as in [MS]).

{\it Let $E\rightarrow X$ be a principal  $G$-bundle,
where
$X$ is a smooth projective variety of
dimension $n$, over a field of characteristic $p >0$. Let
 $P$ denote the Behrend
parabolic of the principal $G$-bundle $E \rightarrow X$, and let $P'$ denote
the
Behrend parabolic of the
 principal $G$-bundle $F^*E\rightarrow X$.  Then
$P' \subseteq P$ if
$$
p \geq \mbox{max}\{(\rank\mg+\dim~X-2),~
\frac{\mu_{max}(\Omega^1_X)}{4}({\rank}\mg)^3,~2\dim~(G/Z(G)), 4h(G)\}.$$}

We can replace $G$ by $G/Z(G)$.
Let $\mu_{max}(\Omega^1_X) \leq 0$. Then, for $L$ = Levi subgroup of $P$,
the prinicipal $L$-bundle $E_L\rightarrow X$ is strongly semistable, by
Theorem 4.1 of [MS]. Hence, by definition, $P$ is the strong Behrend
parabolic. In particular $P' = P$. So we can assume that
$\mu_{max}(\Omega^1_X) > 0$.

Therefore,
 by Proposition~3.4 of [MS], for the Lie algebras  ${\mathfrak{p}}$ 
and ${\mathfrak{p}}'$  associated to $P$ and $P'$ respectively, we have
$$ E(\mg)_0 = E_{P}({\mathfrak{p}}), ~~~\mbox{and}~~(F^*E)(\mg)_0 =
(F^*E)_{P'}(\mathfrak{p'}),$$
where
\begin{equation}\label{eee2}
0\subset E(\mg)_{-r}  \subset\cdots \subset E(\mg)_{-1} \subset
E(\mg)_{0}
\subset~E(\mg)_{1} \subset\cdots \subset E(\mg)_{s} = E(\mg)
\end{equation}
and

\begin{equation}\label{e6}0\subset U_{-m_1}\subset
U_{-m_1+1}\subset
\cdots \subset U_0 = (F^*E)(\mg)_0 \subset
\cdots \subset U_{m_2+1} = (F^*E)(\mg)\end{equation}
are  the HN filtrations of the vector bundles $E(\mg)$ and  $(F^*E)(\mg)$ respectively, such that

$$\mu(E(\mg)_{i}/E(\mg)_{i-1}) < 0, ~\mbox{for}~ i \geq 1
~~\mbox{and}~~\mu(E(\mg)_{i}/E(\mg)_{i-1})\geq 0, ~~~\mbox{for}~~i\leq
0.$$
and similarly

$$\mu(U_j/U_{j-1}) < 0, ~~~\mbox{for }~~ j\geq 1~~~\mbox{and}~~
\mu(U_j/U_{j-1}) \geq 0,~~~\mbox{for}~~j\leq 0.$$
As $p \geq \mbox{max}\{(\rank\mg+\dim~X-2),
~\frac{1}{4}\mu_{max}(\Omega^1_X)(\rank\mg)^3\}$,
by Theorem~\ref{t1}, the
Frobenius pullback of the filtration  (\ref{eee2}) is a subfiltration of
 (\ref{e6}), {\it i.e.},
$$\mbox{for}~~ i\in \{-r,\ldots,  s-1\},  ~~\mbox{there exists}~~
j\in \{-m_1,\ldots,
m_2\} ~~\mbox{such that}~~ F^*E(\mg)_i = U_j.$$

\noindent{\bf Claim}.~$U_0 \subseteq F^*E(\mg)_0$.

\vspace{5pt}
\noindent{Proof of the claim}:~
Suppose $U_0 \nsubseteq F^*E(\mg)_0$. Then let $i > 0$ be
the least integer such that $U_0
\subseteq F^*E(\mg)_i$. This gives a nonzero map of bundles
$$ U_0 \rightarrow F^*E(\mg)_i/F^*E(\mg)_{i-1}.$$
Therefore, by Lemma~\ref{l3}, we have
$$\mu_{min}(U_0) \leq
\mu_{max}(F^*E(\mg)_i/F^*E(\mg)_{i-1}).$$
 We note that,
$$\mu(F^*E(\mg)_0/ F^*E(\mg)_{-1}) =
p\mu(E(\mg)_0/ E(\mg)_{-1}),$$ where,
by Proposition~3.6 of [MS],
for the nilradical $\mathfrak{n}$ of $\mathfrak{p}$, we have,  
$E(\mg)_0/E(\mg)_{-1}  = E_P(\mathfrak{p}/\mathfrak{n})$ and therefore this 
is a vector
bundle of degree 0. By a similar argument, we have
$\mu(U_0/U_{-1}) = 0$.

But $$\mu_{max}\left(\frac{F^*E(\mg)_i}{F^*E(\mg)_{i-1}}\right)
\leq
\mu_{max}\left(\frac{F^*E(\mg)_1}{F^*E(\mg)_0}\right)
< \mu_{min}\left(\frac{F^*E(\mg)_0}{F^*E(\mg)_{-1}}\right)\leq
\mu\left(\frac{F^*E(\mg)_0}{F^*E(\mg)_{-1}}\right) = 0,$$
where the first and second inequalities follow because, as mentioned
before, the Frobenius pull back
of the filtration
(\ref{eee2}) is a subfiltration of the HN filtration (\ref{e6}). This
implies  that $\mu(U_0/U_1) = \mu_{min}(U_0) < 0$, which is a
contradiction. Hence the claim.
Now
$$U_0 = (F^*E)(\mg)_0 \subseteq F^*E(\mg)_0 $$ implies that
 $$(F^*E)_{P'}(\mathfrak{p}') \subseteq F^*E_P(\mathfrak{p}).$$
 This implies that 
$\mathfrak{p}'\subseteq \mathfrak{p}$ and therefore $P'\subseteq
P$.
This completes the proof.

 \end{document}